\documentclass[a4paper,12pt,final]{amsart}
\usepackage{times,a4wide,mathrsfs}

\newcommand{\D}{\mathbb{D}}
\newcommand{\C}{\mathbb{C}}
\newcommand{\RR}{\mathbb{R}}
\newcommand{\lone}{\vert\hspace{-.5mm}\vert\hspace{-.5mm}\vert}
\newcommand{\rone}{\vert\hspace{-.5mm}\vert\hspace{-.5mm}\vert_1}
\newcommand{\Hi}{{\mathscr{H}}^\infty}
\newcommand{\im}{\operatorname{i}}

\newtheorem{theorem}{Theorem}
\newtheorem{lemma}{Lemma}
\newtheorem{corollary}{Corollary}

\begin{document}

\title[The Bohnenblust--Hille inequality is hypercontractive]{The
Bohnenblust--Hille inequality for \\ homogeneous polynomials is
hypercontractive}

\author[A. Defant]{Andreas Defant}
\address{Institut of mathematics, Carl von Ossietzky University, D-26111,
Oldenburg, Germany}
\email{defant@mathematik.uni-oldenburg.de}

\author[L. Frerick]{Leonhard Frerick}
\address{Fachbereich IV - Mathematik, Universit\"{a}t Trier, D-54294 Trier}
\email{frerick@uni-trier.de}

\author[J. Ortega-Cerd\`a]{Joaquim Ortega-Cerd\`a}
\address{Dept.\ Matem\`atica Aplicada i An\`alisi,
Universitat  de Barcelona, Gran Via 585, 08071 Barce- lona, Spain}
\email{jortega@ub.edu}

\author[M. Ouna\"{\i}es]{Myriam Ouna\"{\i}es}
\address{Institut de Recherche Math\'ematique Avanc\'ee,
Universit\'e de Strasbourg, 7 Rue Ren\'e Des\-car\-tes,
67084 Strasbourg CEDEX, France}
\email{ounaies@math.u-strasbg.fr}

\author[K. Seip]{Kristian Seip}
\address{Department of Mathematical Sciences,
Norwegian University of Science and  Technology, NO--7491
Trondheim, Norway} \email{seip@math.ntnu.no} \keywords{
Homogeneous polynomials, Bohnenblust--Hille inequality, Sidon constant, Bohr
radius, power series, Dirichlet polynomials, Dirichlet series}

\subjclass[2000]{32A05, 43A46}

\date{\today}

\thanks{The first author is supported by the German Research Foundation grant AOBJ:561427, 
the third author is supported by the project
MTM2008-05561-C02-01,
and the fifth author is supported by the Research Council of
Norway grant 185359/V30.} \maketitle

\begin{abstract}
The Bohnenblust--Hille inequality says that the
$\ell^{\frac{2m}{m+1}}$-norm of the coefficients of an
$m$-homogeneous polynomial $P$ on $\C^n$ is bounded by $\| P
\|_\infty$ times a constant independent of $n$,  where $\|\cdot
\|_\infty$ denotes the supremum norm on the polydisc $\D^n$. The
main result of this paper is that this inequality is
hypercontractive, i.e., the constant can be taken to be $C^m$ for
some $C>1$. Combining this improved version of the
Bohnenblust--Hille inequality with other results, we obtain the
following: The Bohr radius for the polydisc $\D^n$ behaves
asymptotically as $\sqrt{(\log n)/n}$ modulo a factor bounded away
from $0$ and infinity, and the Sidon constant for the set of
frequencies $\bigl\{ \log n:\ n \ \text{a positive integer}\ \le
N\bigr\}$ is $\sqrt{N}\exp\{(-1/\sqrt{2}+o(1))\sqrt{\log N\log\log
N}\}$ as $N\to \infty$.
\end{abstract}

\section{Introduction and statement of results}

In 1930, Littlewood \cite{Li30} proved the following, often
referred to as Littlewood's $4/3$-inequality: For every bilinear
form $B : \C^n \times  \C^n \rightarrow \mathbb{C}$  we have
\begin{equation*} \label{little}
\bigg( \sum_{i,j} | B(e^{(i)}, e^{(j)}) |^{4/3} \bigg)^{3/4} \leq \sqrt{2}
\sup_{z^{(1)},z^{(2)}\in \D^n}|B(z^{(1)},z^{(2)})|\,,
\end{equation*}
where $\D^n$ denotes the open unit polydisc in $\C^n$ and
$\{e^{(i)}\}_{i=1,\ldots,n}$ is the canonical base of $\C^n$. The exponent
$4/3$ is optimal, meaning that for smaller exponents it will not be
possible to replace $\sqrt{2}$ by a constant independent of $n$.
H.~Bohnenblust and E.~Hille immediately realized the importance of
this result, as well as the techniques used in its proof, for what
was known as Bohr's  absolute convergence problem: Determine the
maximal width $T$ of the vertical strip in which a Dirichlet series
$\sum_{n=1}^\infty a_n n^{-s}$ converges uniformly but not
absolutely. The problem was raised by H.~Bohr \cite{B13b} who in
1913 showed that $T \leq 1/2$. It remained a central problem in the
study of Dirichlet series until 1931, when Bohnenblust and Hille
\cite{BH31} in an ingenious way established that $T=1/2$.

A crucial ingredient in \cite{BH31} is an $m$-linear version of
Littlewood's $4/3$-inequality: For each
$m$ there is a constant $C_m \geq 1$ such that for every
$m$-linear form $B : \C^n \times \cdots \times \C^n \rightarrow \mathbb{C}$
we have
\begin{equation} \label{eq:Bohn-Hille}
\bigg( \sum_{i_1, \ldots,i_m } |B(e^{(i_{1})}, \ldots   ,
e^{(i_{m})})|^{\frac{2m}{m+1}} \bigg)^{\frac{m+1}{2m}} \leq C_m \sup_{z^{(i)}
\in \D^n}|B(z^{(1)}, \dots, z^{(m)})|\,,
\end{equation}
and  again the exponent $\frac{2m}{m+1}$ is optimal. Moreover, if
$C_m$ stands for the best constant, then the original proof gives
that  $C_m \leq m^{^\frac{m+1}{2m}} (\sqrt{2})^{m-1}$. This
inequality was long forgotten and rediscovered more than forty years
later by A.~Davie \cite{Da73} and S.~Kaijser \cite{Ka78}. The proofs
in \cite{Da73} and \cite{Ka78} are slightly different from the
original one and give the better estimate
\begin{equation}
\label{constant}
C_m \leq (\sqrt{2})^{m-1}\,.
\end{equation}

In order to solve Bohr's absolute convergence problem, Bohnenblust and Hille
needed a symmetric version of \eqref{eq:Bohn-Hille}. For this purpose, they in
fact invented polarization and deduced from \eqref{eq:Bohn-Hille} that for each
$m$ there is a constant $D_m\geq 1$ such for every $m$-homogeneous polynomial
$\sum_{|\alpha| = m} a_\alpha z^\alpha$ on $\C^n$
\begin{equation} \label{poly}
\big( \sum_{|\alpha| = m} |a_\alpha|^{\frac{2m}{m+1}} \big)^{\frac{m+1}{2m}}
\leq   D_m
 \sup_{z \in \D^n} \Bigl|\sum_{|\alpha| = m} a_\alpha z^\alpha\Bigr|\,;
\end{equation}
they showed again, through a highly nontrivial argument, that the exponent
$\frac{2m}{m+1}$ is best possible.

Let us assume that $D_m$ in \eqref{poly} is optimal. By an
estimate of L.~A.~Harris \cite{Harris75} for the polarization
constant of $\ell^\infty$, getting from \eqref{constant} to
\begin{equation*} \label{constantsym}
D_m \leq( \sqrt{2} )^{m-1} \frac{m^{\frac
m2}(m+1)^{\frac{m+1}{2}}}{2^{m} (m!)^{\frac{m+1}{2m}}}
\end{equation*}
is now quite straightforward; see e.~g. \cite[Section 4]{DeSe}.
Using Sawa's Khinchine-type inequality for Steinhaus variables,
H.~Queff{\'e}lec \cite[Theorem III-1]{Queffelec} obtained the
slightly better estimate
\begin{equation}\label{queffelec}
D_m \leq \big( \frac{2}{\sqrt{\pi}} \big)^{m-1} \frac{m^{\frac
m2}(m+1)^{\frac{m+1}{2}}}{2^{m} (m!)^{\frac{m+1}{2m}}}.
\end{equation}
Our main result is that the Bohnenblust--Hille inequality
\eqref{poly} is in fact hypercontractive, i.e.,
$D_m\le C^m$ for some $C\ge 1$:

\begin{theorem} \label{mainresult}
Let $m$ and $n$ be positive integers larger than $1$. Then we have
\begin{equation} \label{maininequality}
\bigl( \sum_{|\alpha| = m} |a_\alpha|^{\frac{2m}{m+1}} \bigr)^{\frac{m+1}{2m}}
\leq \Bigl(1+\frac 1{m-1}\Bigr)^{m-1}\sqrt{m} (\sqrt{2})^{m-1} \sup_{z \in \D^n}
\Bigl|\sum_{|\alpha| = m}
a_\alpha z^\alpha\Bigr|
\end{equation}
for every $m$-homogeneous polynomial $\sum_{|\alpha| = m} a_\alpha z^\alpha$ on
$\C^n$.
\end{theorem}
Before presenting the proof of this theorem, we mention some
particularly interesting consequences that serve to illustrate its
applicability and importance.

We begin with the Sidon constant $S(m,n)$ for the index set
$\{\alpha=(\alpha_1, \alpha_2,\ldots,\alpha_n): \ |\alpha|=m\}$,
which is defined in the following way. Let
\[ P(z)=\sum_{|\alpha|=m} a_\alpha z^\alpha \]
be an $m$-homogeneous polynomial in $n$ complex
variables. We set
\[\| P \|_\infty=\sup_{z\in \D^n} |P(z)| \qquad \text{and} \qquad
 \lone P \rone=\sum_{|\alpha|=m}
|a_\alpha|;\] then $S(m,n)$ is the smallest constant $C$ such that
the inequality $ \lone P\rone\le C \| P \|_\infty$ holds for every
$P$. It is plain that $S(1,n)=1$ for all $n$, and this case is
therefore excluded from our discussion. Since the dimension of the space of
$m$-homogeneous polynomials in $\C^n$ is $\binom{n+m+1}{m}$, an application of
H\"older's inequality to \eqref{maininequality} gives:

\begin{corollary} \label{sidon}
Let $m$ and $n$ be positive integers larger than $1$. Then
\begin{equation} \label{sharp} S(m,n) \le \Bigl(1+\frac 1{m-1}\Bigr)^{m-1} \sqrt
m (\sqrt 2)^{m-1}
\binom{n+m-1}{m}^{\frac{m-1}{2m}}.
\end{equation}
\end{corollary}

Note  that  the Sidon constant $S(m,n)$ coincides with the unconditional basis
constant of the monomials $z^\alpha$ of degree $m$ in
$H^\infty(\D^n)$ which is defined as the best constant $C \geq 1$ such
that for every $m$-homogeneous polynomial $\sum_{|\alpha| = m}
a_\alpha z^\alpha$ on $\D^n$ and any choice of scalars
$\varepsilon_\alpha$ with $|\varepsilon_\alpha| \leq 1$ we have
\[
\sup_{z \in \D^n} \Bigl|\sum_{|\alpha| = m} \varepsilon_\alpha a_\alpha
z^\alpha\Bigr|
\leq C \sup_{z \in \D^n} \Bigl|\sum_{|\alpha| = m} a_\alpha z^\alpha\Bigr|\,.
\]
This and similar
unconditional basis constants were studied in \cite{DeDiGaMa}, where it was
established that $S(m,n)$ is bounded from above and below by $n^{\frac{m-1}2}$
times constants depending only on $m$. The more precise estimate
\begin{equation} \label{old} S(m,n) \le C^m n^{\frac{m-1}{2}},
\end{equation} with $C$ an absolute constant, can be extracted from
\cite{DefFrer06}.

Note that we also have the following trivial estimate:
\begin{equation} \label{trivial}
S(m,n) \le \sqrt{\binom{n+m-1}{m}},
\end{equation}
which is a consequence of the Cauchy--Schwarz inequality along with the fact
that the number of different monomials of degree $m$ in $n$ variables is
$\binom{n+m-1}{m}$. Comparing \eqref{sharp} and \eqref{trivial}, we see that our
estimate gives a nontrivial result only in the range $\log n >m$. Using the
Salem--Zygmund inequality for random trigonometric polynomials (see \cite[p.
68]{Kahane}), one may check that we have obtained the right value for $S(m,n)$,
up to a factor less than $c^m$ with $c<1$ an absolute constant
(for a different argument see \cite[(4.4)]{DeGarMa_BorhLoc}).

We will use our estimate for $S(m,n)$ to find the
precise asymptotic behavior of the $n$-dimen-sional Bohr radius,
which was introduced and studied by H.~Boas and D.~Khavinson
\cite{BoasKhav97}. Following \cite{BoasKhav97}, we now let $K_n$ be the largest
positive number $r$ such that all polynomials $\sum_\alpha
a_\alpha z^\alpha$ satisfy
\[
\sup_{z\in r \D^n}\sum_\alpha |a_\alpha z^\alpha| \le \sup_{z\in
\D^n}\Bigl|\sum_\alpha a_\alpha z^\alpha\Bigr|.
\]
The classical Bohr radius $K_1$ was studied and estimated by
H.~Bohr \cite{Bohr14} himself, and it was shown independently by
M.~Riesz, I.~Schur, and F.~Wiener  that $K_1=1/3$. In
\cite{BoasKhav97}, the two inequalities
\begin{equation}\label{BoasKhav}
\frac {1}{3} \sqrt{\frac{1}{n}} \le K_n \le 2 \sqrt{\frac{\log
n}{n}}
\end{equation}
were established for $n>1$. The paper of Boas and Khavinson
aroused new interest in the Bohr radius and has been a source of
inspiration for many subsequent papers. For some time (see for
instance \cite{Boas00}) it was thought that the left-hand side of
\eqref{BoasKhav} could not be improved. However, using
\eqref{old}, A.~Defant and L.~Frerick \cite{DefFrer06} showed that
\[
K_n\ge c \sqrt{\frac{\log n}{n\log\log n}}
\]
holds for some absolut constant $c >0$.

Using Corollary~\ref{sidon}, we will prove the following estimate which in view of \eqref{BoasKhav} is asymptotically
optimal.
\begin{theorem}\label{bohrradius}
The $n$-dimensional Bohr radius $K_n$ satisfies
\[
K_n\ge \gamma \sqrt {\frac{\log n}{n}}
\]
for an absolute constant $\gamma>0$.
\end{theorem}
Combining this result with the right inequality in
\eqref{BoasKhav}, we conclude that \begin{equation} \label{bn}
K_n=b(n) \sqrt{\frac{\log n}{n}} \end{equation} with $\gamma\le
b(n)\le 2$. We will in fact obtain
\[
b(n)\ge \frac{1}{\sqrt 2}+o(1)
\]
when $n\to\infty$ as a lower estimate; see the concluding remark
of Section~4, which contains the proof of
Theorem~\ref{bohrradius}.

Using a different argument, Defant and Frerick have also computed
the right asymptotics for the Bohr radius for the unit ball in $\mathbb C^n$
with the $\ell^p$ norm. This result will be presented in the forthcoming paper
\cite{DefFrer}.

Another interesting point is that Theorem~\ref{mainresult} yields a
refined version of a striking theorem of S.~Konyagin and
H.~Queffel\'{e}c \cite[Theorem 4.3]{KQ01} on Dirichlet polynomials, a
result that was recently sharpened by R. de la Bret\`{e}che
\cite{Br08}. To state this result, we define the Sidon constant
$S(N)$ for the index set
\[ \Lambda(N)=\bigl\{ \log n:\ n \ \text{a positive integer}\ \le N\bigr\}
\]
in the following way. For a Dirichlet polynomial
\[ Q(s)=\sum_{n=1}^N a_n n^{-s},  \]
we set $\|Q\|_\infty =\sup_{t\in \RR} |Q(it)|$ and $\lone Q \rone
= \sum_{n=1}^N |a_n|$. Then $S(N)$ is the smallest constant $C$
such that the inequality $\lone Q \rone \le C \|Q\|_\infty$ holds
for every $Q$.
\begin{theorem}\label{dirichletpol} We have
\begin{equation}\label{sharpSidon} S(N)=\sqrt{N}
\exp\Bigl\{\bigl(-\frac{1}{\sqrt{2}}+o(1)\bigl)\sqrt{\log N
\log\!\log N}\Bigr\}
\end{equation}
when $N\to \infty$.
\end{theorem}
The inequality
\[
S(N)\ge \sqrt{N}
\exp\Bigl\{\bigl(-\frac{1}{\sqrt{2}}+o(1)\bigl)\sqrt{\log N
\log\!\log N}\Bigr\} \] was established by R. de la Bret\`{e}che
\cite{Br08} combining methods from analytic number theory with
probabilistic arguments. It was also shown in \cite{Br08} that the
inequality
\[
S(N)\le\sqrt{N} \exp\Bigl\{\bigl(-\frac{1}{2\sqrt{2}}+o(1)\bigl)\sqrt{\log N
\log\log N}\Bigr\}
\]
follows from an ingenious method developed by Konyagin and
Queff\'{e}lec in \cite{KQ01}. The same argument, using
Theorem~\ref{mainresult} instead of the weaker inequality
\eqref{queffelec}, gives \eqref{sharpSidon}. More precisely,
following Bohr, we set $z_j=p_j^{-s}$, where $p_1,p_2,...$ denote
the prime numbers ordered in the usual way, and make accordingly a
translation of Theorem~\ref{mainresult} into a statement about
Dirichlet polynomials; we then replace Lemme 2.4 in \cite{Br08} by
this version of Theorem~\ref{mainresult} and otherwise follow the
arguments in Section 2.2 of \cite{Br08} word by word.

Theorem~\ref{dirichletpol} enables us to make a nontrivial remark on Bohr's
absolute convergence problem. To this end, we recall that a theorem of Bohr
\cite{B13a} says that the abscissa of uniform convergence equals the abscissa of
boundedness and regularity for a given Dirichlet series $\sum_{n=1}^\infty a_n
n^{-s}$; the latter is the infimum of those $\sigma_0$ such that the function
represented by the Dirichlet series  is analytic and bounded in $\Re
s=\sigma>\sigma_0$. When discussing Bohnenblust and Hille's solution of Bohr's
problem, it is therefore quite natural to introduce the space $\Hi$, which
consists of those bounded analytic functions $f$ in $\C_+=\{s=\sigma+\im t:\
\sigma>0\}$ such that $f$ can be represented by an ordinary Dirichlet series
$\sum_{n=1}^\infty a_n n^{-s}$ in some half-plane.

\begin{corollary}\label{bh}
The supremum of the set of real numbers $c$ such that
\begin{equation}\label{balacalquef}
\sum_{n=1}^\infty |a_n|\, n^{-\frac{1}{2}} \exp\Bigl\{ c\sqrt{\log n\log \log
n}\Bigr\}<\infty
\end{equation}
for every $\sum_{n=1}^\infty a_n n^{-s}$ in $\Hi$ equals $1/\sqrt{2}$.
\end{corollary}

This result is a refinement of a theorem of R.~Balasubramanian,
B.~Calado, and H.~Queff\'{e}lec \cite[Theorem 1.2]{BCQ06}, which
implies that \eqref{balacalquef} holds for every $\sum_{n=1}^\infty a_n
n^{-s}$ in $\Hi$ if $c$ is sufficiently small. We will present the
deduction of Corollary~\ref{bh} from Theorem~\ref{dirichletpol} in
Section~\ref{last} below.

An interesting consequence of the theorem of Balasubramanian,
Calado, and Queff\'{e}lec is that the Dirichlet series of an element
in $\Hi$ converges absolutely on the vertical line $\sigma=1/2$. But
Corollary~\ref{bh} gives a lot more; it adds a level precision that
enables us to extract much more precise information about the
absolute values $|a_n|$ than what is obtained from the solution of
Bohr's absolute convergence theorem.

\section{Preliminaries on multilinear forms}
We begin by fixing some useful index sets. For two positive
integers $m$ and $n$, both assumed to be larger than $1$, we
define \[M(m,n)=\Big\{i=(i_1,\ldots,i_m):\ \ i_1,\ldots,i_m
\in\{1,\ldots,n\}\Big\}\]  and
\[J(m,n)=\Big\{j=(j_1,\ldots,j_m)\!\in {M(m,n)}:\ \  j_1\leq \cdots\leq
j_m \Big\}.\] For indices $i,j\in M(m,n)$, the notation $i\sim j$
will mean that there is a permutation $\sigma$ of the set
$\{1,2,\ldots,m\}$ such that $i_{\sigma(k)}=j_k$ for every
$k=1,\ldots,m$. For a given index $i$, we denote by $[i]$ the
equivalence class of all indices $j$ such that $i\sim j$.
Moreover, we let $|i|$ denote the cardinality of $[i]$ or in other
words the number of different indices belonging to $[i]$. Note
that for each $i \in M(m,n)$ there is a unique $j\in J(m,n)$ with
$[i]=[j]$. Given an index $i$ in $M(m,n)$, we set $i^k =
(i_1,\ldots,i_{k-1},i_{k+1}, \ldots, i_m)$, which is then an index
in $M(m-1,n)$.

The transformation of a homogeneous polynomial to a corresponding
multilinear form will play a crucial role in the proof of
Theorem~\ref{mainresult}. We denote by $B$ an $m$-multilinear form
on $\C^n$, i.e., given $m$ points $z^{(1)},\ldots,z^{(m)}$ in
$\C^n$, we set
\[
B(z^{(1)},\ldots,z^{(m)})=\sum_{i\in M(m,n)} b_i z_{i_1}^{(1)}\cdots
z_{i_m}^{(m)}.
\]
We may express the coefficients as $b_i=B(e^{(i_1)},\ldots,e^{(i_m)})$. The form
$B$ is symmetric if for every permutation $\sigma$ of the set
$\{1,2,\ldots,m\}$,
$B(z^{(1)},\ldots,z^{(m)})=B(z^{(\sigma(1))},\ldots,z^{(\sigma(m))})$. If we
restrict a symmetric multilinear form to the diagonal $P(z)=B(z,\ldots,z)$, then
we obtain a homogeneous polynomial. The converse is also true: Given a
homogeneous polynomial $P:\C^n\to \C$ of degree $m$, by polarization, we may
define the symmetric m-multilinear form $B: \C^{n}\times\cdots\times\C^n\to \C$
so that $B(z,\ldots,z)=P(z)$. In what follows, $B$ will denote the symmetric
$m$-multilinear form obtained in this way from $P$.

It will be important for us to be able to relate the norms of $P$
and $B$. It is plain that $\|P\|_\infty=\sup_{z\in \D^n} |P(z)|$ is smaller
than $\sup_{\D^n\times\cdots\times\D^n}|B|$. On the
other hand, it was proved by Harris \cite{Harris75} that we
have, for non-negative integers $m_1,\ldots,m_k$ with
$m_1+\cdots+m_k=m$,
\begin{equation}\label{harris}
|B(\underbrace{z^{(1)},\ldots,z^{(1)}}_{m_1},\ldots,\underbrace{z^{(k)},\ldots,
z^{(k)}}_{m_k})|\le \frac{m_1!\cdots m_k!}{m_1^{m_1}\cdots
m_k^{m_k}}\frac{m^m}{m!}\
\|P\|_\infty.
\end{equation}

Given an $m$-homogeneous polynomial in $n$ variables
$P(z)=\sum_{|\alpha|=m}a_\alpha z^\alpha$, we will write it as
\[
P(z)=\sum_{j\in J(m,n)} c_j z_{j_1}\cdots z_{j_m}.
\]
For every $i$ in $M(m,n)$, we set $c_{[i]}=c_j$ where $j$ is the unique element
of $J(m,n)$ with $i\sim j$. Observe that in this representation the coefficient
$b_i$ of the multilinear form $B$ associated to $P$ can be computed from its
corresponding coefficient: $b_i=c_{[i]}/|i|$.

\section{Proof of Theorem~\ref{mainresult}}\label{mainsect}

For the proof of  Theorem \ref{mainresult}, we will need two
lemmas. The first is due to R.~Blei \cite[Lemma 5.3]{Blei79}:
\begin{lemma}\label{Blei}
For all families $(c_i)_{i \in M(m,n)}$ of complex numbers, we
have
\begin{equation*}
\Big(\sum_{i \in M(m,n)} |c_i|^{\frac{2m}{m+1}}
\Big)^{\frac{m+1}{2m}} \leq \prod_{1 \leq k \leq m} \Big[\sum_{i_k
= 1}^n \Big(\sum_{i^k \in M(m-1,n)} |c_i|^2
\Big)^\frac{1}{2}\Big]^{\frac{1}{m}}.
\end{equation*}

\end{lemma}
We now let $\mu^n$ denote normalized Lebesgue measure on $\mathbb
T^n$; the second lemma is a result of F.~Bayart \cite[Theorem
9]{Bayart02}, whose proof relies on an inequality first
established by A.~Bonami \cite[Th\'{e}or\`{e}me 7, Chapitre
III]{Bo}.
\begin{lemma}\label{Bayart}
For every $m$-homogeneous polynomial $P(z) =\sum\limits_{|\alpha|=m}
a_\alpha z^\alpha$ on $\C^n$, we have
\[
\big( \sum_{|\alpha| = m} |a_\alpha|^2 \big)^{\frac{1}{2}} %= \Big\|
%\sum_{|\alpha|=m} a_\alpha z^\alpha\Big\|_{L_2 (\mu^n)}
\le (\sqrt{2})^m \, \Big\| \sum_{|\alpha|=m} a_\alpha
z^\alpha\Big\|_{L^1 (\mu^n)}.
\]
\end{lemma}
We note also that
Lemma~\ref{Bayart} is a special case of a variant of Bayart's
theorem found in \cite{Helson}, relying on an inequality in
D.~Vukotic's paper \cite{Vukotic}. The latter inequality, giving the
best constant in an inequality of Hardy and Littlewood, had appeared
earlier in a paper of M.~Mateljevi\'{c} \cite{Ma80}.

\begin{proof}[Proof of Theorem~\ref{mainresult}]
We write  the homogeneous polynomial $P$ as
\[P(z)=\sum_{j\in J(m,n)}c_jz_{j_1}\cdots z_{j_m}.\]
We now get
\[
\sum_{j\in J(m,n)}|c_j|^{\frac{2m}{m+1}}= \sum_{i\in M(m,n)}
|i|^{-\frac1{m+1}}\Bigl(\frac{|c_{[i]}|}{|i|^{\frac12}}\Bigr)^{\frac{2m}{m+1}}
\le\sum_{i\in M(m,n)}
\Bigl(\frac{|c_{[i]}|}{|i|^{\frac12}}\Bigr)^{\frac{2m}{m+1}}.
\]
Using Lemma~\ref{Blei} and the estimate $|i|/|i_k|\le m$, we
therefore obtain
\[
\begin{split}
\Bigl(\sum_{j\in
J(m,n)}|c_j|^{\frac{2m}{m+1}}\Bigr)^{\frac{m+1}{2m}} & \le \
\prod_{k=1}^m\ \Bigl[\sum_{i_k=1}^m \Bigl(\sum_{i^k\in
M(m-1,n)}\frac{|c_{[i]}|^2}{|i|}\Bigr)^{\frac12} \Bigr]^{\frac1m}
\\ & \le \ \sqrt{m} \prod_{k=1}^m \ \Bigl[\sum_{i_k=1}^n
\Bigl(\sum_{i^k\in
M(m-1,n)}|i^k|\frac{|c_{[i]}|^2}{|i|^2}\Bigr)^{\frac12}
\Bigr]^{\frac1m}.
\end{split}
\]
Thus it suffices to prove that
\begin{equation}\label{suffices}
\sum_{i_k=1}^n \Bigl(\sum_{i^k\in
M(m-1,n)}|i^k|\frac{|c_{[i]}|^2}{|i|^2}\Bigr)^{\frac12} \le \Bigl(1+\frac
1{m-1}\Bigr)^{m-1}
(\sqrt{2})^{m-1} \|P\|_\infty
\end{equation}
for $k=1,2,\ldots,m$.

 We observe that if we
write $P_k(z)=B(z,\ldots,z,e^{(i_k)},z,\ldots,z)$, then we have
\[
 \Bigl(\sum_{i^k\in M(m-1,n)}|i^k|\frac{|c_{[i]}|^2}{|i|^2}\Bigr)^{\frac12}=
\Bigl(\sum_{i^k\in M(m-1,n)} |i^k|
|b_i|^2\Bigr)^{\frac12}=\|P_k\|_{2}.
\]
Hence, applying Lemma~\ref{Bayart} to $P_k$, we get
\[
 \Bigl(\sum_{i^k\in M(m-1,n)}|i^k|\frac{|c_{[i]}|^2}{|i|^2}\Bigr)^{\frac12}
\le (\sqrt{2})^{m-1} \int_{\mathbb T^n}
|B(z,\ldots,z,e^{(i_k)},z,\ldots,z)| \ d\mu^n(z).
\]
It is clear that we may replace $e^{(i_k)}$ by $\lambda_{i_k}(z)
e^{(i_k)}$ with $\lambda_{i_k}(z)$ any point on the unit circle. If
we choose $\lambda_{i_k}(z)$ such that $B(z,\ldots,z,
\lambda_{i_k}(z)e^{(i_k)},z,\ldots,z) >0$ and write
$\tau_k(z)=\sum_{i_k=1}^n \lambda_{i_k}(z) e^{(i_k)}$, then we
obtain
\[
\sum_{i_k=1}^n \Bigl(\sum_{i^k\in
M(m-1,n)}|i^k|\frac{|c_{[i]}|^2}{|i|^2}\Bigr)^{\frac12}
\le(\sqrt{2})^{m-1}
 \int_{\mathbb T^n}
B(z,\ldots,z,\tau_k(z),z,\ldots,z)\ d\mu(z).
\]
We finally arrive at \eqref{suffices} by applying \eqref{harris} to the
right-hand side of this inequality.
\end{proof}

\section{Proof of Theorem~\ref{bohrradius}}\label{sectradius}

We now turn to multidimensional Bohr radii. In
\cite[Theorem~2.2]{DeGarMa_BorhLoc}, a basic link between Bohr
radii and unconditional basis constants was given. Indeed,  we
have
\[
  \frac{1}{3\sup_m \sqrt[m]{C_m}}\,
  \le K_n \le \min \Bigl(\frac{1}{3}, \; \frac{1}{\sup_m \sqrt[m]{C_m}}\Bigr)\,,
\]
where  $C_m$ is the
unconditional basis constant of the monomials of degree $m$ in $H^\infty(\D^n)$.
Thus the estimates
for unconditional basis constants for $m$-homoge\-neous polynomials
always lead to estimates for multidimensional Bohr radii.

We still choose to present a direct proof of
Theorem~\ref{bohrradius}, as this leads to a better estimate on the
asymptotics of the quantity $b(n)$ in \eqref{bn}. We need the
following lemma of F.~Wiener (see \cite{BoasKhav97}).
\begin{lemma}\label{FWiener}
Let $P$ be a polynomial in $n$ variables and $P=\sum_{m\ge 0} P_m$
its expansion in homogeneous polynomials. If $\|P\|_\infty \le 1$,
then $\|P_m\|_\infty \le 1-|P_0|^2$ for every $m>0$.
\end{lemma}

 \begin{proof}[Proof of Theorem~\ref{bohrradius}] We assume that $\sup_{\D^n}
\bigl|\sum a_\alpha z^\alpha \bigr|\le 1$. Observe that for all $z$ in
$r\D^n$,
\[
\sum |a_\alpha z^\alpha|\le |a_0|+\sum_{m>1} r^m\sum_{|\alpha|=m}|a_\alpha|.
\]
If we take into account the estimates
\[
\frac{(\log n)^{m}}{n}\le m!\qquad\text{and}\qquad
\binom{n+m-1}{m}\le e^m \bigr(1+\frac{n}{m}\bigl)^m,
\]
 then Corollary~\ref{sidon} and Lemma~\ref{FWiener} give
%\begin{equation}\label{newsum}
\[ \sum_{m>1} r^m\sum_{|\alpha|=m}|a_\alpha| \le \sum_{m>1} r^m
e\sqrt{m} (2\sqrt e)^m\Bigl(\frac{n}{\log n}\Bigr)^{m/2}
(1-|a_0|^2).
\]
%\end{equation}
Choosing $r\le \varepsilon \sqrt{\frac{\log n}{n}}$ with $\varepsilon$ small
enough, we obtain
\[
\sum |a_\alpha z^\alpha|\le |a_0| + (1-|a_0|^2)/2\le 1
\]
whenever $|a_0|\le 1$. Thus the theorem is proved with $\gamma=\varepsilon$.
\end{proof}

A closer examination of this proof shows that we get a better
constant if in the range $m> \log n$ we use \eqref{trivial} instead
of Corollary~\ref{sidon}. By this approach, we get
\[
b(n)\ge \frac{1}{\sqrt{2}}+o(1)
\]
when $n\to \infty.$

\section{Proof of Corollary~\ref{bh}}\label{last}

We need the following auxiliary result \cite[Lemma 1.1]{BCQ06}.

\begin{lemma}\label{queffelec}
If $f(s)=\sum_{n=1}^\infty a_n n^{-s}$ belongs to $\Hi$, then we
have
\begin{equation}\label{contour} \Bigl\|\sum_{n=1}^N a_n n^{-s}\Bigr\|_\infty\le
C  \log N \sup_{t>0} |f(\sigma+i t)|
\end{equation}
for an absolute constant $C$ and every $N\ge 2$.
\end{lemma}

\begin{proof}[Proof of Corollary~\ref{bh}]
For this proof, we will use the notation $n_k=2^k$. Assume first
that $c<1/\sqrt{2}$, and suppose we are given an arbitrary element
$f(s)=\sum_{n=1}^\infty a_n n^{-s}$ in $\Hi$. Then we have
\[ \sum_{n=1}^\infty |a_n|\, n^{-\frac{1}{2}} \exp\Bigl\{ c\sqrt{\log n\log \log
n}\Bigr\}\le \sum_{k=0}^\infty n_k^{-\frac{1}{2}} \exp\Bigl\{
c\sqrt{\log n_k\log \log n_k}\Bigr\} \sum_{n=1}^{n_{k+1}}|a_n|.\]
Applying Theorem~\ref{dirichletpol} and Lemma~\ref{queffelec} to
each of the sums $\sum_{n=1}^{n_{k+1}} |a_n|$, we see that the
right-hand is finite.

On the other hand, assume instead that $c>1/\sqrt{2}$. By
Theorem~\ref{dirichletpol}, we may find a positive constant $\delta$
and a sequence of Dirichlet polynomials
\[Q_k(s)=\sum_{n=1}^{n_{2k}-1} a_n^{(k)} n^{-s} \]
such that $\| Q_k\|_\infty =1$ and
\[ \sum_{n=1}^{n_{2k}-1} |a_n^{(k)}| \ge \delta
\exp\Bigl\{-c\sqrt{\log n_{2k} \log\!\log n_{2k}}\Bigr\} \] for
$k=1,2,...$. In fact, by the construction in \cite[Section
2.1]{Br08}, we may assume that \begin{equation}\label{dyad}
\sum_{n=n_{2(k-1)}}^{n_{2k}-1} |a_n^{(k)}| \ge \delta
\exp\Bigl\{-c\sqrt{\log n_{2k} \log\!\log n_{2k}}\Bigr\}
\end{equation} for $k=1,2,...$. We observe that the function
\[ f(s)=\sum_{k=1}^\infty \exp\Bigl\{-\varepsilon\sqrt{\log n_{2k}
\log\!\log n_{2k}}\Bigr\}\ Q_k(s)\] is an element in $\Hi$ for every
positive $\varepsilon$. Setting $f(s)=\sum_{n=1}^\infty a_n n^{-s}$
and assuming again that $Q_k$ has been constructed as in
\cite[Section 2.1]{Br08}, we get that
\[ \sum_{n=n_{2(k-1)}}^{n_{2k}-1} |a_n| \ge C \sum_{n=n_{2(k-1)}}^{n_{2k}-1}
|a_n^{(k)}| \] for some constant $C$ independent of $k$ and
$\varepsilon$. (Here the point is that $a_n^{(j)}$ decays
sufficiently fast when $j$ grows because $n_{2(j+1)}=4 n_{2j}$.)
Combining this estimate with \eqref{dyad}, we see that
\[ \sum_{n=1}^\infty |a_n| \exp\Bigl\{\bigl( c+\varepsilon\bigr)\sqrt{\log n\log \log
n}\Bigr\}=\infty. \] Since this can be achieved for arbitrary
$c>1/\sqrt{2}$ and $\varepsilon>0$, the result follows.
\end{proof}

\end{document}